\documentclass[10pt,oneside]{amsart}

\usepackage[
paper=a4paper,
headsep=20pt,text={136mm,208mm},centering,includehead
]{geometry}

\usepackage[bookmarks]{hyperref}
\usepackage{amssymb,amsxtra,dsfont}
\usepackage[all]{xy}
\usepackage{pb-diagram,pb-xy}
\usepackage{graphicx,color}
\usepackage{epstopdf}
\usepackage{pinlabel}
\usepackage{float}
\usepackage{array,calc,booktabs}
\usepackage[
  textwidth=1.2in,
  backgroundcolor=yellow,
  bordercolor=orange,
  textsize=small
]{todonotes}
\usepackage{tikz}

\iftrue
\makeatletter
\def\@maketitle{%
  \normalfont\normalsize
  \@adminfootnotes
  \@mkboth{\@nx\shortauthors}{\@nx\shorttitle}%
  \global\topskip42\p@\relax 
  \@settitle
  \ifx\@empty\authors \else \@setauthors \fi
  \ifx\@empty\@dedicatory
  \else
    \baselineskip22\p@
    \vtop{{\small\itshape\@dedicatory\@@par}%
      \global\dimen@i\prevdepth}\prevdepth\dimen@i
  \fi
  \@setabstract
  \normalsize
  \if@titlepage
    \newpage
  \else
    \dimen@25\p@ \advance\dimen@-\baselineskip
    \vskip\dimen@\relax
  \fi
} 
\def\@settitle{%
  \vspace*{-15pt}
  \begin{flushleft}%
    \LARGE\bfseries
    \strut\@title\strut
  \end{flushleft}%
}
\def\@setauthors{%
  \begingroup
  \def\thanks{\protect\thanks@warning}%
  \trivlist
  \raggedright
  \large \@topsep27\p@\relax
  \advance\@topsep by -\baselineskip
  \item\relax
  \author@andify\authors
  \def\\{\protect\linebreak}%
  \authors
  \ifx\@empty\contribs
  \else
    ,\penalty-3 \space \@setcontribs
    \@closetoccontribs
  \fi
  \normalfont
  \endtrivlist
  \endgroup
}
\def\@setaddresses{\par
  \nobreak \begingroup
  \small\raggedright
  \def\author##1{\nobreak\addvspace\smallskipamount}%
  \def\\{\unskip, \ignorespaces}%
  \interlinepenalty\@M
  \def\address##1##2{\begingroup
    \par\addvspace\bigskipamount\noindent
    \@ifnotempty{##1}{(\ignorespaces##1\unskip) }%
    {\ignorespaces##2}\par\endgroup}%
  \def\curraddr##1##2{\begingroup
    \@ifnotempty{##2}{\nobreak\noindent\curraddrname
      \@ifnotempty{##1}{, \ignorespaces##1\unskip}\/:\space
      ##2\par}\endgroup}%
  \def\email##1##2{\begingroup
    \@ifnotempty{##2}{\nobreak\noindent E-mail address%
      \@ifnotempty{##1}{, \ignorespaces##1\unskip}\/:\space
      \ttfamily##2\par}\endgroup}%
  \def\urladdr##1##2{\begingroup
    \def~{\char`\~}%
    \@ifnotempty{##2}{\nobreak\noindent\urladdrname
      \@ifnotempty{##1}{, \ignorespaces##1\unskip}\/:\space
      \ttfamily##2\par}\endgroup}%
  \addresses
  \endgroup
  \global\let\addresses=\@empty
}
\def\@setabstracta{%
    \ifvoid\abstractbox
  \else
    \skip@15pt \advance\skip@-\lastskip
    \advance\skip@-\baselineskip \vskip\skip@
    \box\abstractbox
    \prevdepth\z@ 
    \vskip-15pt
  \fi
}
\renewenvironment{abstract}{%
  \ifx\maketitle\relax
    \ClassWarning{\@classname}{Abstract should precede
      \protect\maketitle\space in AMS document classes; reported}%
  \fi
  \global\setbox\abstractbox=\vtop \bgroup
    \normalfont\small
    \list{}{\labelwidth\z@
      \leftmargin0pc \rightmargin\leftmargin
      \listparindent\normalparindent \itemindent\z@
      \parsep\z@ \@plus\p@
      
    }%
    \item[\hskip\labelsep\bfseries\abstractname.]%
}{%
  \endlist\egroup
  \ifx\@setabstract\relax \@setabstracta \fi
}

\def\ps@headings{\ps@empty
  \def\@evenhead{%
    \setTrue{runhead}%
    \normalfont\scriptsize
    \rlap{\thepage}\hfill
    \def\thanks{\protect\thanks@warning}%
    \leftmark{}{}}%
  \def\@oddhead{%
    \setTrue{runhead}%
    \normalfont\scriptsize
    \def\thanks{\protect\thanks@warning}%
    \rightmark{}{}\hfill \llap{\thepage}}%
  \let\@mkboth\markboth
}\ps@headings

\def\section{\@startsection{section}{1}%
  \z@{-1.4\linespacing\@plus-.5\linespacing}{.8\linespacing}%
  {\normalfont\bfseries\Large}}
\def\subsection{\@startsection{subsection}{2}%
  \z@{-.8\linespacing\@plus-.3\linespacing}{.5\linespacing\@plus.2\linespacing}%
  {\normalfont\bfseries\large}}
\def\subsubsection{\@startsection{subsubsection}{3}%
  \z@{.7\linespacing\@plus.2\linespacing}{-1.5ex}%
  {\normalfont\bfseries}}
\def\@secnumfont{\bfseries}

\renewcommand\contentsnamefont{\bfseries}
\def\@starttoc#1#2{\begingroup
  \setTrue{#1}%
  \par\removelastskip\vskip\z@skip
  \@startsection{}\@M\z@{\linespacing\@plus\linespacing}%
    {.5\linespacing}{
      \contentsnamefont}{#2}%
  \ifx\contentsname#2%
  \else \addcontentsline{toc}{section}{#2}\fi
  \makeatletter
  \@input{\jobname.#1}%
  \if@filesw
    \@xp\newwrite\csname tf@#1\endcsname
    \immediate\@xp\openout\csname tf@#1\endcsname \jobname.#1\relax
  \fi
  \global\@nobreakfalse \endgroup
  \addvspace{32\p@\@plus14\p@}%
  \let\tableofcontents\relax
}
\def\contentsname{Contents}
\def\l@section{\@tocline{2}{.5ex}{0mm}{5pc}{}}
\def\l@subsection{\@tocline{2}{0pt}{2em}{5pc}{}}
\makeatother
\fi 

\def\to{\mathchoice{\longrightarrow}{\rightarrow}{\rightarrow}{\rightarrow}}
\makeatletter
\newcommand{\shortxra}[2][]{\ext@arrow 0359\rightarrowfill@{#1}{#2}}
\def\longrightarrowfill@{\arrowfill@\relbar\relbar\longrightarrow}
\newcommand{\longxra}[2][]{\ext@arrow 0359\longrightarrowfill@{#1}{#2}}

\def\addtagsub#1{\let\oldtf=\tagform@\def\tagform@##1{\oldtf{##1}\hbox{$_{#1}$}}}

\makeatother

\makeatletter
\def\Nopagebreak{\@nobreaktrue\nopagebreak}

\newtheoremstyle{theorem-giventitle}
        {}{}              
        {\itshape}                      
        {}                              
        {\bfseries}                     
        {.}                             
        {\thm@headsep}                             
        {\thmnote{\bfseries#3}}
\newtheoremstyle{theorem-givenlabel}
        {}{}              
        {\itshape}                      
        {}                              
        {\bfseries}                     
        {.}                             
        {\thm@headsep}                             
        {\thmname{#1}~\thmnumber{#3}\setcurrentlabel{#3}}

\newtheoremstyle{definition-giventitle}
        {}{}              
        {}                      
        {}                              
        {\bfseries}                     
        {.}                             
        {\thm@headsep}                             
        {\thmnote{\bfseries#3}}
\def\setcurrentlabel#1{\gdef\@currentlabel{#1}}

\makeatother

\newtheorem{theorem}{Theorem}[section]

\newtheorem{lemma}[theorem]{Lemma}

\theoremstyle{definition}

\newtheorem{remark}[theorem]{Remark}

\theoremstyle{theorem-giventitle}
\newtheorem{theorem-named}{}
\theoremstyle{theorem-givenlabel}
\newtheorem{theorem-labeled}{Theorem}

\theoremstyle{definition-giventitle}
\newtheorem{definition-named}{}
\newtheorem{step-named}{}

\numberwithin{equation}{section}

\def\Z{\mathbb{Z}}
\def\Q{\mathbb{Q}}

\def\NN{\mathbb{N}}

\def\N{\mathcal{N}}

\def\F{\mathcal{F}}

\def\cA{\mathcal{A}}
\def\cC{\mathcal{C}}

\def\Ker{\operatorname{Ker}}
\def\Coker{\operatorname{Coker}}

\def\sign{\operatorname{sign}}
\def\rank{\operatorname{rank}}

\def\Arf{\operatorname{Arf}}

\def\lsign{\sign^{(2)}}

\def\Lt{L^2}
\def\rhot{\rho^{(2)}}

\def\setminus{\smallsetminus}

\begin{document}

\vspace*{0mm}

\title[Knots having the same Seifert form]{Knots having the same Seifert form and primary decomposition of knot concordance}

\author{Taehee Kim}
\address{
  Department of Mathematics\\
  Konkuk University \\
  Seoul 05029\\
  Korea
}
\email {tkim@konkuk.ac.kr}

\def\subjclassname{\textup{2010} Mathematics Subject Classification}
\expandafter\let\csname subjclassname@1991\endcsname=\subjclassname
\expandafter\let\csname subjclassname@2000\endcsname=\subjclassname


\begin{abstract}
We show that for each Seifert form of an algebraically slice knot with nontrivial Alexander polynomial, there exists an infinite family of knots having the Seifert form such that the knots are linearly independent in the knot concordance group and not concordant to any knot with coprime Alexander polynomial.  Key ingredients for the proof are Cheeger--Gromov--von Neumann $\rhot$-invariants for amenable groups developed by Cha--Orr and polynomial splittings of metabelian  $\rhot$-invariants.
\end{abstract}

\maketitle


\section{Introduction}

A knot is \emph{slice} if it bounds a locally flat 2-disk in the 4-ball, and two knots $K$ and $J$ are \emph{concordant} if $K\# (-J)$ is slice. The concordance classes of knots form an abelian group under connected sum. This abelian group is called \emph{the knot concordance group}, which we denote by $\cC$. There is a surjective homomorphism from $\cC$ to the algebraic concordance group of Seifert forms which sends the concordance class of a knot to the algebraic concordance class of a Seifert form of the knot. It is known by Levine \cite{Levine:1969-2, Levine:1969-1} and Stoltzfus \cite{Stoltzfus:1977-1}  that the algebraic concordance group is isomorphic to $\Z^\infty\oplus \Z_2^\infty \oplus \Z_4^\infty$. The kernel of the above surjection is the subgroup of (the concordance classes of) \emph{algebraically slice knots}, denoted $\cA$. The classification of the group $\cA$ (and $\cC)$ is yet unknown, and in this paper we address the structure of $\cA$ related to Seifert forms and the Alexander polynomial. Note that a knot with trivial Alexander polynomial is slice by the work of Freedman \cite{Freedman:1982-1, Freedman-Quinn:1990-1}. 

\begin{theorem}[Main Theorem]
\label{theorem:main-theorem}
	Let $V$ be a Seifert form of an algebraically slice knot $K$ with nontrivial Alexander polynomial $\Delta_K(t)$. Then there exists an infinite family of knots $\{K_i\}_{i\in \NN}$ which satisfies  the following:
	\begin{enumerate}
		\item the knots $K_i$ have the Seifert form $V$, 
		\item the knots $K_i$ are linearly independent in $\cC$,
		\item for each $i$ and nonzero integer $n$,  the knot $nK_i$ is not concordant to \emph{any} knot whose Alexander polynomial is coprime to $\Delta_K(t)$.
	\end{enumerate}
\end{theorem}

We review known results on the structure of knot and link concordance under fixed Alexander invariants and primary decomposition of $\cC$.
The aforementioned work of Freedman is equivalent to that if $\Delta_K(t)=1$, then $K$ is concordant to the unknot. Namely, trivial Alexander polynomial determines a unique concordance class. A natural question arises asking if there is any other Alexander polynomial or a Seifert form which determines a unique concordance class.  
This question was answered in the negative by Livingston \cite{Livingston:2002-1} using Casson--Gordon invariants under a certain condition on Seifert forms. Later, using Cheeger--Gromov--von Neumann $\rhot$-invariants, the author removed the condition on Seifert forms and gave the following theorem:

\begin{theorem}\label{theorem:old_main theorem}\cite{Kim:2005-1}
Let $V$ be a Seifert form of a knot $K$ with nontrivial Alexander polynomial. Then, there exists an infinite family of knots $\{K_i\}_{i\in \NN}$ such that $K_i$ have the Seifert form $V$ and are pairwise nonconcordant.   
\end{theorem}   
In \cite{Kim:2005-1} it was not shown that the $K_i$ in Theorem~\ref*{theorem:old_main theorem} are linearly independent in $\cC$, and Theorems~\ref{theorem:main-theorem}(1) and (2) extend Theorem~\ref{theorem:old_main theorem} for the case of Seifert forms of algebraically slice knots by giving examples which are linearly independent in $\cC$. 

Theorem~\ref{theorem:old_main theorem} was extended in various directions. Cochran and the author \cite{Cochran-Kim:2004-1} gave an infinite family of pairwise nonconcordant knots having the same \emph{higher-order} Alexander invariants, and it was extended further in \cite{Kim:2016-1} so that the knots are linearly independent in $\cC$. Cha, Friedl, and Powell \cite{Cha-Friedl-Powell:2014-1} generalized Theorem~\ref{theorem:old_main theorem} to link concordance. Recently Kauffman and Lopes gave infinitely many \emph{nonisotopic} pretzel knots with the same Alexander invariants \cite{Kauffman-Lopes:2017-1}.

Theorem~\ref{theorem:main-theorem}(3) is related to primary decomposition of knot concordance. A theorem of Levine \cite{Levine:1969-2} playing an essential role in classification of the algebraic concordance group is that if the connected sum of two knots with coprime Alexander polynomials is algebraically slice, then so are the knots. A similar decomposition in $\cC$ or $\cA$ is unknown, and we have the following open question: if two knots $K$ and $J$ have coprime Alexander polynomials and the connected sum $K\# J$ is slice, then are $K$ and $J$ slice? Put differently, if $K$ and $J$ have coprime Alexander polynomials and any of $K$ and $J$ is not slice, then is $K$ nonconcordant to $J$?

Regarding the above question, Se-Goo Kim \cite{Kim:2005-2} showed splittings of Casson--Gordon invariants for knots with coprime Alexander polynomials. A similar polynomial splitting property of the metabelian Cheeger--Gromov--von Neumann $\rhot$-invariants was shown by  Se-Goo Kim and the author \cite{Kim-Kim:2008-1} (see Theorem~\ref{theorem:splitting}),  and it was extended to splittings of \emph{higher-order} $\rhot$-invariants \cite{Kim-Kim:2014-1}. On smooth concordance, similar polynomial splittings of $d$-invariants on slicing knots \cite{Bao:2015-1} and doubly slicing knots \cite{Kim-Kim:2016-1} were also shown. 

In \cite{Kim:2005-2, Kim-Kim:2008-1}, the examples of knots which are not concordant to any knot with coprime Alexander polynomial were given, but they were constructed for \emph{some} prescribed Alexander modules. For instance, the examples in \cite{Kim-Kim:2008-1} have Alexander modules which have a unique nontrivial proper submodule.  On the other hand, Theorem~\ref{theorem:main-theorem} gives examples for \emph{any} Seifert form of an algebraically slice knot with nontrivial Alexander polynomial. 

There is the solvable filtration $\{\F_n\}$ of $\cC$ defined in \cite{Cochran-Orr-Teichner:1999-1}, which is indexed by nonnegative half-integers. A notable property of $\{\F_n\}$ is that all metabelian sliceness obstructions, including Casson--Gordon invariants, vanish for knots in $\F_n$ when $n\ge 1.5$. Nevertheless,  for each $n$, Cochran, Harvey, and Leidy \cite{Cochran-Harvey-Leidy:2009-2} gave a similar primary decomposition of a family of knots in $\F_n$ constructed using \emph{robust doubling operators} (see Definitions~4.4 and 7.2 and Theorem~7.7 in \cite{Cochran-Harvey-Leidy:2009-2}). Also, there is a similar primary decomposition of a family of order 2 elements in $\F_n$ \cite{Cochran-Harvey-Leidy:2009-3, Jang:2015-1}. In \cite{Cochran-Harvey-Leidy:2009-2, Cochran-Harvey-Leidy:2009-3, Jang:2015-1},  the examples of knots were shown to be nonconcordant to any knot  with coprime Alexander polynomial which is constructed using doubling operators and Arf invariant zero knots (for instance, see \cite[Theorem~6.2]{Cochran-Harvey-Leidy:2009-2}).
 
To construct $K_i$ in Theorem~\ref{theorem:main-theorem}, we use (iterated) satellite construction. To show their linear independence in $\cC$, we use Cheeger--Gromov--von Neumann $\rhot$-invariants for amenable groups, which were developed by Cha and Orr \cite{Cha-Orr:2009-1} on homology cobordism, and later adapted to knot concordance by Cha \cite{Cha:2010-1} (see Theorem~\ref{theorem:obstruction}). To show Theorem~\ref{theorem:main-theorem}(3), we use polynomial splittings of metabelian $\rhot$-invariants in \cite{Kim-Kim:2008-1} (see Theorem~\ref{theorem:splitting}). 

This paper is organized as follows. We review necessary results on $\rhot$-invariants in Section~\ref{section:preliminaries}. In Section~\ref{section:proof}, we give a proof of Theorem~\ref{theorem:main-theorem}. In this paper, homology groups are with integer coefficients unless specified otherwise. By abuse of notation, we use the same symbol for a knot and its homology and homotopy classes. For a prime $p$, we denote the field of $p$ elements by $\Z_p$. All manifolds are assumed to be oriented and compact. 
\subsection*{Acknowledgments} 
This research was supported by Basic Science Research Program through the National Research Foundation of Korea(NRF) funded by the Ministry of Education (no. 2011-0030044(SRC-GAIA) and no. 2015R1D1A1A01056634).

\section{Preliminaries}\label{section:preliminaries}
In this section, we review necessary results on $\rhot$-invariants in \cite{Cochran-Orr-Teichner:1999-1, Cha:2010-1, Kim-Kim:2008-1}.

Let $M$ be a closed 3-manifold and $\phi\colon \pi_1M\to \Gamma$ a homomorphism to a countable (discrete) group $\Gamma$. By enlarging the group $\Gamma$ if necessary, we may assume that there exists a 4-manifold $W$ with $\partial W=M$ such that $\phi$ extends to $\tilde{\phi}\colon \pi_1W \to \Gamma$. Then, \emph{the Cheeger--Gromov--von-Neumann $\rhot$-invariant} associated with $(M,\phi)$ \cite{Cheeger-Gromov:1985-1} can be defined to be the $\Lt$-signature defect as follows:
\[
\rhot(M, \phi):=\lsign_\Gamma(W) - \sign (W).
\]
In the above, $\sign(W)$ is the ordinary signature of $W$, and $\lsign_\Gamma(W)$ is the $\Lt$-signature of the intersection form on $H_2(W;\N\Gamma)$ where $\N\Gamma$ denotes the group von Neumann algebra of $\Gamma$. We refer the reader to \cite[Section~2]{Cha:2010-1} for more details on $\rhot$-invariants. Based on the work on $\rhot$-invariants in \cite{Cha-Orr:2009-1}, Cha obtained the following sliceness obstruction, which extends the sliceness obstruction using $\rhot$-invariants in \cite{Cochran-Orr-Teichner:1999-1}. In the following, $M(K)$ denotes the zero-framed surgery on a knot $K$ in $S^3$. 

\begin{theorem}\label{theorem:obstruction}\cite[Theorem~1.2]{Cha:2010-1}
	Suppose $K$ is a slice knot and $\Gamma$ is an amenable group lying in Strebel'€™s class $D(R)$ for some ring $R$. If $\phi\colon \pi_1M(K) \to \Gamma$ is a homomorphism extending to a slice disk exterior for $K$, then $\rhot(M(K),\phi)=0$.
\end{theorem} 

One can find the definitions of amenable group and Strebel's class $D(R)$ in \cite{Cha-Orr:2009-1}, but they will not be needed in this paper; we will only need Lemma~\ref{lemma:amenable group} below.

For a group $G$, let $G^{(1)}:= [G,G]$, the commutator subgroup of $G$, and let $G^{(2)}:=[G^{(1)},G^{(1)}]$. For a prime $p$, we also define
\[
G^{(2)}_p:=\Ker\{G^{(1)}\to (G^{(1)}/G^{(2)})\otimes \Z_p\}.
\]

\begin{lemma}\label{lemma:amenable group} Suppose $G$ is a group with $H_1(G)\cong \Z$. Then, for each prime $p$, the group $G/G^{(2)}_p$ is amenable and lies in Strebel's class $D(\Z_p)$. 
\end{lemma}
\begin{proof}
Since $G/G^{(1)}\cong H_1(G)\cong \Z$ and $G^{(1)}/G_p^{(2)}$ injects into $(G^{(1)}/G^{(2)})\otimes \Z_p$, the groups $G^{(1)}/G^{(2)}$ and $G^{(1)}/G_p^{(2)}$ are abelian and have no torsion coprime to $p$. Now the conclusion follows from \cite[Lemma~6.8]{Cha-Orr:2009-1}.
\end{proof}

We review a vanishing criterion for metabelian $\rhot$-invariants for slice knots in\cite{Cochran-Orr-Teichner:1999-1}. For a knot $K$, there is \emph{the rational Blanchfield form} 
\[
B\ell\colon H_1(M(K);\Q[t^{\pm 1}])\times H_1(M(K);\Q[t^{\pm 1}]) \to \Q(t)/\Q[t^{\pm 1}].
\]
For a $\Q[t^{\pm 1}]$-module $P$ of $H_1(M(K);\Q[t^{\pm 1}])$, we define
\[
P^\perp:= \{x\in H_1(M(K);\Q[t^{\pm 1}])\,\mid\, B\ell(x,y)=0 \mbox{ for all }y\in P\}.
\]
If $P=P^\perp$, we say that $P$ is \emph{self-annihilating} with respect to the rational Blanchfield form. 

Letting the group $\Z=H_1(M(K))=\langle t\rangle$ act on $H_1(M(K);\Q[t^{\pm 1}])$ via the action of $t$, we obtain the semi-direct product $H_1(M(K);\Q[t^{\pm 1}])\rtimes \Z$. Then, each element $x\in H_1(M(K);\Q[t^{\pm 1}])$ induces a homomorphism
\[
\phi_x\colon\pi_1M(K) \to H_1(M(K);\Q[t^{\pm 1}])\rtimes \Z \to \Q(t)/\Q[t^{\pm 1}]\rtimes \Z
\]
such that $\phi_x(y) = (B\ell(x, y\mu^{-\epsilon(y)}), \epsilon(y))$ where $\mu$ is the meridian of $K$ and $\epsilon\colon \pi_1M(K) \to \Z=H_1(M(K))$ is the abelianization (see \cite[Section~3]{Cochran-Orr-Teichner:1999-1}).

We say that $K$ has \emph{vanishing metabelian $\rhot$-variants} if there exists a self-annihilating submodule $P$ with respect to the rational Blanchfield form such that $\rhot(M(K), \phi_x) = 0$ for all $x\in P$. We have the following theorem.
\begin{theorem}\label{theorem:vanishing-metabelian-rho-invariants}\cite[Theorem 4.6]{Cochran-Orr-Teichner:1999-1} 
A slice knot has vanishing metabelian $\rhot$-invariants. 
\end{theorem}

\section{Proof of Theorem~\ref{theorem:main-theorem}}\label{section:proof}
\subsubsection*{Construction of the $K_i$}
Let $K$ be an algebraically slice knot with $\Delta_K(t)\ne 1$ which has a Seifert form $V$. Since there exists a slice knot having the same Seifert form as $K$ (for instance, see \cite[Proposition~12.2.1]{Kawauchi:1996-1}), we may assume that $K$ is slice. 

We will construct the desired $K_i$ using (iterated) satellite construction. We briefly explain satellite construction we will use in this paper. Let $\eta_1, \eta_2, \ldots, \eta_m$ be simple closed curves in $S^3\setminus K$ such that the curves $\eta_\ell$ form an unlink in $S^3$. Let $J$ be a knot. Now take the union of $S^3\setminus N(\eta_1)$ and $S^3\setminus N(J)$ along their common boundary $S^1\times S^1$ via an orientation reversing homeomorphism such that a meridian (resp. 0-framed longitude) of $\eta_1$ is identified with a zero-framed longitude (resp. a meridian) of $J$. Iterating this process, for each $\ell=1,2,\ldots, m$, replace the open tubular neighborhood $N(\eta_\ell)$ of $\eta_\ell$ with the exterior of $J$. The resulting ambient space is homeomorphic to $S^3$, and the image of $K$ under this process becomes a new knot in $S^3$, which we denote by $K(\eta_1,\ldots, \eta_m;J)$ or $K(\eta_\ell; J)$ for simplicity. We will construct $K_i$ as $K(\eta_\ell; J_i)$ form some choice of $\eta_\ell$ and $J_i$, where the choice of $\eta_1,\ldots, \eta_m$ will be independent of $i$.   

We choose $\eta_1,\ldots, \eta_m$ for $K_i$ as follows. Let $F$ be a Seifert surface for $K$ with which the Seifert form $V$ is associated. Considering $F$ as a disk with $2g$ bands added, take $\eta_\ell$ to be the curves dual to the bands of $F$ (hence $m=2g$). 

Since $H_1(M(K);\Z_p[t^{\pm 1}])\cong H_1(M(K);\Z[t^{\pm 1}])\otimes \Z_p$ and $\Z_p$ is a field, it easily follows that the $\eta_\ell$ generate $H_1(M(K);\Z_p[t^{\pm 1}])$ for each prime $p$. It is well-known that the $\eta_\ell$ also  generate $H_1(M(K);\Q[t^{\pm 1}])$. This is a key property of the $\eta_\ell$ which we will use later.

We explain how to choose $J_i$. In \cite{Cheeger-Gromov:1985-1}, it was shown that there exists a constant $C_K$ such that $|\rhot(M(K), \phi)|<C_K$ for \emph{every} homomorphism $\phi\colon \pi_1M\to \Gamma$ where $\Gamma$ is a countable group. (One can take an explicit value for $C_K$ as $69713280\cdot c(K)$ where $c(K)$ is the crossing number of $K$ \cite[Theorem~1.9]{Cha:2014-1}.) Now we choose $J_i$ to be the knots in Lemma~\ref{lemma:J_i} below. For a knot $K$, let $a_K$ be the top coefficient of $\Delta_K(t)$ and $\sigma_K$ the Levine-Tristram signature function for $K$.
\begin{lemma}\label{lemma:J_i} For the constants $C_K$ and $a_K$ defined as above, there exists a sequence of knots $J_1, J_2,\ldots$ and a sequence of primes $p_1, p_2, \ldots$ which satisfy the following:
\begin{enumerate}
	\item $\Arf(J_i)=0$ for each $i$ and $a_K<p_1 <p_2 <p_3 < \cdots$,
	\item $\sum_{r=0}^{p_i-1}\sigma_{J_i}(e^{2\pi r\sqrt{-1}/p_i})>p_iC_K$ for all $i$,
	\item $\sum_{r=0}^{p_i-1}\sigma_{J_j}(e^{2\pi r\sqrt{-1}/p_i})=0$ for $j>i$,
	\item $\int_{S^1}\sigma_{J_i}(\omega)\,\, d\omega > C_K$ for all $i$.
\end{enumerate}
\end{lemma}
\begin{proof}
Let $\{p_i\}$ be any increasing sequence of primes bigger than $a_K$. Let $w_i :=e^{2\pi\sqrt{-1}/p_i}$. By \cite[Lemma~5.6]{Cha:2007-2}, for each $i$ there exists a knot $L_i$ and neighborhoods $N(\omega_i)$ and $N(\omega_i^{-1})$ of $\omega_i$ and $\omega_i^{-1}$, respectively, which are disjoint from $\omega_j^r$ for all $j<i$ and all $r\in \Z$ such that $\sigma_{L_i}$ is positive inside $N(\omega_i)\cup N(\omega_i^{-1})$ and 0 outside $N(\omega_i)\cup N(\omega_i^{-1})$. Now for each $i$, the desired knot $J_i$ can be obtained by taking the connected sum of sufficiently many even number of copies of $L_i$.
\end{proof}

Now for each $i$ we define $K_i:=K(\eta_\ell;J_i)$ where $\eta_\ell$ and $J_i$ are defined as above. 

\subsubsection*{Proof of Theorem~\ref{theorem:main-theorem}(1)}
Since we have chosen $\eta_\ell$ in the complement of the Seifert surface $F$ for $K$, for each $i$, the image of $F$ under the satellite construction for $K_i$ becomes a Seifert surface for $K_i$ which has the same Seifert form as $F$. This proves Theorem~\ref{theorem:main-theorem}(1). 

\subsubsection*{Proof of Theorem~\ref{theorem:main-theorem}(3)} 
We prove Theorem~\ref{theorem:main-theorem}(3) before proving Theorem~\ref{theorem:main-theorem}(2). Recalling Theorem~\ref{theorem:vanishing-metabelian-rho-invariants} and the fact that a knot and its inverse have the same Alexander polynomial, we have the following theorem on polynomial splittings of metabelian $\rhot$-invariants.
\begin{theorem}\label{theorem:splitting}\cite[Theorem~3.1]{Kim-Kim:2008-1}
	Suppose two knots $K$ and $J$ have coprime Alexander polynomials. If $K$ does not have vanishing metabelian $\rhot$-invariants, then $K$ is not concordant to $J$.
\end{theorem}

Note that since $\Delta_{nK_i}(t) = (\Delta_K(t))^{|n|}$, the Alexander polynomial of a knot is coprime to that of $K$ if and only if it is coprime to that of $nK_i$. Therefore, by Theorem~\ref{theorem:splitting}, to prove Theorem~\ref{theorem:main-theorem}(3) it suffices to show that $nK_i$ does not have vanishing metabelian $\rhot$-invariants for each $n$ and $i$. Fix $n$ and $i$. By taking the inverse of $K$ if necessary, we may assume $n>0$. Suppose to the contrary that $nK_i$ has vanishing metabelian $\rhot$-invariants. Then, there exists a self-annihilating submodule $P$ of $H_1(M(nK_i);\Q[t^{\pm 1}])$ such that $\rhot(M(nK_i), \phi_x)=0$ for all $x\in P$. It is well-known that since $P$ is a self-annihilating submodule, 
\[
\rank_\Q P = \frac12 \rank_\Q H_1(M(nK_i);\Q[t^{\pm 1}]) = \frac12 \deg \Delta_{nK_i}(t).
\]
Since $\Delta_K(t)\ne 1$, it follows that $P\ne 0$. 

Fix $x\in P$ such that $x\ne 0$. In particular, $\rhot(M(nK_i), \phi_x) = 0$ where $\phi_x$ is the homomorphism $\pi_1M(nK_i)\to \Q(t)/\Q[t^{\pm 1}]\rtimes \Z$ induce from $x$ as defined in Section~\ref{section:preliminaries}. We will show that this will lead us to a contradiction.

To compute $\rhot(M(nK_i), \phi_x)$, we construct a cobordism $C$ such that 
\[
\partial C  = M(nK_i)\coprod (-\coprod^nM(K_i))
\]
as follows. Let $C$ be the standard cobordism between $M(nK_i)$ and $\coprod^nM(K_i)$ as in \cite[p.113]{Cochran-Orr-Teichner:2002-1}. Briefly speaking, $C$ is obtained from $\coprod^nM(K_i)\times [0,1]$ by attaching $n-1$ 1-handles whose resulting top boundary is the 0-framed surgery on the split link of $n$ copies of $K_i$, and then attaching $n-1$ 2-handles whose attaching circles are zero-framed circles represented by $\mu_j\mu_{j+1}^{-1}$, respectively, for $1\le j\le n-1$ where $\mu_j$ is the meridian of the $j$th copy of $K_i$. 

Then, one can see that $\pi_1C\cong \pi_1(M(nK_i))/\langle \ell_1,\ldots, \ell_n\rangle$ where $\langle \cdots\rangle$ denotes the normal subgroup generated by $\cdots$ and each $\ell_j$ is the 0-framed longitude of the $j$th copy of $K_i$ (for example, see \cite[Lemma~3.1 and p.810]{Kim-Kim:2014-1}). For simplicity, let $G:=\Q(t)/\Q[t^{\pm 1}]\rtimes \Z$. Since $\ell_j\in \pi_1(M(nK_i))^{(2)}$ for each $j$ and $G^{(2)}= 0$, it follows that $\phi_x(\ell_j) = 0$ for all $j$. Therefore, $\phi_x$ extends to $\pi_1C\to G$, which is also denoted by $\phi_x$.
 
 For each $j$, let $\phi_x^j\colon \pi_1(M(K_i))\to G$ be the restriction of $\phi_x\colon \pi_1C\to G$ to the $j$th copy of $M(K_i)$ in the bottom boundary of $C$. Let $B\ell$ and $B\ell_i$ denote the rational Blanchfield forms of $nK_i$ and $K_i$, respectively. Then, $H_1(M(nK_i);\Q[t^{\pm 1}])\cong \oplus^n H_1(M(K_i);\Q[t^{\pm 1}])$ and $B\ell\cong \oplus^nB\ell_i$. Therefore, we can write $x=(x_1,x_2,\ldots, x_n)$ for some $x_j\in H_1(M(K_i);\Q[t^{\pm 1}])$ for $1\le j\le n$, and then for each $y=(y_1,\ldots, y_n)\in H_1(M(nK_i);\Q[t^{\pm 1}])$ we have $B\ell(x,y) = \sum_{j=1}^n B\ell_i(x_j,y_j)$. Also, we can identify $y_j\in H_1(M(K_i);\Q[t{\pm 1}])$ with $y=(y_1,y_2,\ldots, y_n)\in H_1(M(nK_i);\Q[t^{\pm 1}])$ such that $y_i=0$ for $i\ne j$, and we obtain $B\ell(x,y_j) = B\ell(x,y)= \sum_{j=1}^n B\ell_i(x_j,y_j) = B\ell_i(x_j,y_j)$. Therefore, for each $j$ one can deduce that $\phi_x^j = \phi_{x_j}$, the homomorphism induced from $x_j$. 
 
Now, from the definition of $\rhot$-invariants, we obtain
\[
\lsign_G(C) - \sign(C) = \rhot(M(nK_i),\phi_x) - \sum_{j=1}^n \rhot(M(K_i), \phi_{x_j}).
\]
Using Mayer Vietoris sequences, we can show $H_2(C)\cong H_2(\partial_+C)$ where $\partial_+C:= M(nK_i))$, and hence 
\[
\Coker\{H_2(\partial C_+)\to H_2(C)\}=0.
\]
Therefore, $\sign(C)=0$, and we also obtain $\lsign_G(C)=0$ by \cite[Theorem 6.6]{Cha-Orr:2009-1} (or see the proof of\cite[Lemma~4.2]{Cochran-Orr-Teichner:2002-1}). Therefore, we have
\[
\rhot(M(nK_i),\phi_x) = \sum_{j=1}^n \rhot(M(K_i), \phi_{x_j}).
\]
Since $\rhot(M(nK_i),\phi_x)=0$ by our choice of $x$,  we obtain
\begin{equation}\label{equation:rho-invariant}
\sum_{j=1}^n \rhot(M(K_i), \phi_{x_j})=0.
\end{equation}

We compute $\rhot(M(K_i),\phi_{x_j})$ for each $j$. If $x_j=0$,  then $\rhot(M(K_i),\phi_{x_j})=0$. For, in this case the $\phi_{x_j}$ maps onto $\Z$, and therefore $\rhot(M(K_i),\phi_{x_j})=\int_{S^1}\sigma_{K_i}(\omega)\,d\omega$ (see (2.3) on p.108 and Lemma~5.3 in \cite{Cochran-Orr-Teichner:2002-1}). Since $K_i$ has the same Seifert form $V$ as the slice knot $K$, we have $\int_{S^1}\sigma_{K_i}(\omega)\,d\omega=0$. 

Suppose $x_j\ne 0$. Recall that $K_i= K(\eta_\ell;J_i) = K(\eta_1,\ldots, \eta_m; J_i^1,\ldots, J_i^m)$ where $J_i^\ell$ is the $\ell$th copy of $J_i$ for each $\ell=1,2,\ldots, m$. Since each longitude $\ell_j\in \pi_1(M(K_i))^{(2)}$, the homomorphism $\phi_{x_j}$ uniquely extends to $\pi_1M(K)\to G$ and $\pi_1M(J_i^\ell)\to G$ for $\ell=1,2,\ldots, m$, which we denote by $\phi_j$ and $\phi_j^\ell$, respectively (see \cite[p.1429]{Cochran-Harvey-Leidy:2009-1}). Furthermore, since the meridian of $J_i^\ell$ is identified with the longitude of $\eta_\ell\in \pi_1(M(K))^{(1)}$, the homomorphism $\phi_j^\ell$ maps into $G^{(1)}=\Q(t)/\Q[t^{\pm 1}]$, which is an abelian group. Therefore, we have the following lemma which immediately follows from \cite[Lemma~2.3]{Cochran-Harvey-Leidy:2009-1}. For convenience, let us identify $\eta_\ell$ in $M(K)$ with its image in $M(K_i)$. Note that  $\phi_j(\eta_\ell) = \phi_{x_j}(\eta_\ell)$.
\begin{lemma}\label{lemma:computation-of-rho-invariant}\cite[Lemma~2.3]{Cochran-Harvey-Leidy:2009-1}
In the above setting, we have
\[
\rhot(M(K_i),\phi_{x_j}) = \rhot(M(K),\phi_j) + \sum_{\ell=1}^m\rhot(M(J_i^\ell),\phi_j^\ell),
\]
where
\[
	\rhot(M(J_i^\ell),\phi_j^\ell)= 
	\begin{cases}
	0 & \mbox{if } \phi_{x_j}(\eta_\ell)= 0,\\[1ex]
	\int_{S^1} \sigma_{J_i}(\omega)\,d\omega & \mbox{if } \phi_{x_j}(\eta_\ell)\ne 0.
	
	\end{cases}
\]
\end{lemma}
Since the $\eta_\ell$ $(1\le \ell\le m)$ generate $H_1(M(K_i);\Q[t^{\pm 1}])\cong H_1(M(K);\Q[t^{\pm 1}])$ and the rational Blanchfield form $B\ell_i$ is nonsingular, there exists at least one $\ell$ such that $B\ell_i(x_j,\eta_\ell)\ne 0$ and hence $\phi_{x_j}(\eta_\ell)\ne 0$. Since $\int_{S^1} \sigma_{J_i}(\omega)\,d\omega>0$ by Lemma~\ref{lemma:J_i}(4), from Lemma~\ref{lemma:computation-of-rho-invariant} we deduce that $\rhot(M(K_i),\phi_{x_j})\ge -C_K +\int_{S^1} \sigma_{J_i}(\omega)\,d\omega$. 

Summarizing the computations, we have
\[
\rhot(M(K_i),\phi_{x_j})
\begin{cases}
=0 & \mbox{if } x_j= 0,\\[1ex] 
 \ge -C_K +\int_{S^1} \sigma_{J_i}(\omega)\,d\omega & \mbox{if } x_j\ne 0.
\end{cases}
\]

Let $d$ be the number of $j$ such that $x_j\ne 0$. Now we obtain that
\[
\sum_{j=1}^n \rhot(M(K_i), \phi_{x_j}) \ge d \left(-C_K + \int_{S^1} \sigma_{J_i}(\omega)\,d\omega\right), 
\]
Since $x\ne 0$, we have $d> 0$. By Lemma~\ref{lemma:J_i}(4), it follows that 
\[
\sum_{j=1}^n \rhot(M(K_i), \phi_{x_j}) > 0,
\]
which contradicts Equation~(\ref{equation:rho-invariant}). This proves Theorem~\ref{theorem:main-theorem}(3). 

\subsubsection*{Proof of Theorem~\ref{theorem:main-theorem}(2)}
We show that $K_i$ are linearly independent in $\cC$, namely, no nontrivial linear combination of $K_i$ are slice. 
This can be easily shown by following the arguments in the proof of \cite[Theorem~4.2]{Kim:2016-1}. Moreover, a proof of Theorem~\ref{theorem:main-theorem}(2) is easier than that of \cite[Theorem~4.2]{Kim:2016-1} in the sense that it does not need the technicalities used in the proof of \cite[Theorem~4.2]{Kim:2016-1} such as modules over noncommutative rings and the notion of algebraic $n$-solutions. For the reader's convenience, we adapt the proof of \cite[Theorem~4.2]{Kim:2016-1} to our case, and give a proof of Theorem~\ref{theorem:main-theorem}(2) below.

Suppose to the contrary that $L:=\#_i a_iK_i$ $(a_i\in \Z)$, a nontrivial connected sum of finitely many copies of $\pm K_i$, is slice. We may assume $a_1\ne 0$ by reindexing, and by taking the inverse of $L$ if necessary we may assume further that $a_1>0$. We construct a 4-manifold $W$ by stacking up the following building blocks $V$, $C$, and $V_i$. For a 4-manifold $X$ and a homomorphism $\phi\colon \pi_1X\to \Gamma$ where $\Gamma$ is a group, for simplicity let $S_\Gamma(X):=\lsign_\Gamma(X)-\sign(X)$.
\begin{enumerate}
	\item Let $V$ be the exterior of a slice disk for $L$ in $D^4$. Then, $\partial V = M(L)$.
	\item Let $C$ be the standard cobordism between $M(L)$ and $\coprod a_iM(K_i)$ as constructed in the proof of Theorem~\ref{theorem:main-theorem}(3). Turning $C$ upside down, we may assume $\partial C = (\coprod a_iM(K_i))\coprod (-M(L))$. 	
	\item For each $i$, let $V_i$ be the 4-manifold with $\partial V_i= M(K_i)$ given by \cite[Lemma~4.1(1)]{Kim:2016-1} satisfying the following: suppose $\phi\colon \pi_1V_i\to \Gamma$ is a homomorphism where $\Gamma$ is an amenable group lying in Strebel's class $D(R)$ for some ring $R$. Let $d_\ell$ be the order of $\phi(\eta_\ell)$ in $\Gamma$ and let $\phi_\ell\colon \pi_1M(J_i)\to \Z_{d_\ell}$ be an epimorphism sending the meridian of $J_i$ to $1\in \Z_{d_\ell}$ (where $\Z_\infty:=\Z$).  Then, $S_\Gamma(V_i)=\sum_{\ell=1}^m\rhot(M(J_i),\phi_\ell)$. 
	\item Let $U$ be the 4-manifold with $\partial U=M(K)\coprod (-M(K_1))$ which is given by \cite[Lemma~4.1(2)]{Kim:2016-1} satisfying the following: suppose $\phi\colon \pi_1U \to \Gamma$ is a homomorphism where $\Gamma$ is a group as in (3). Let $d_\ell$ and  $\phi_\ell\colon \pi_1M(J_1)\to \Z_{d_\ell}$ be as in (3). Then, $S_\Gamma(U)=-\sum_{\ell=1}^m\rhot(M(J_1),\phi_\ell)$. 
\end{enumerate}
Let $b_1:=a_1-1$ and $b_i=|a_i|$ for $i\ge 2$. For each $i\ge 1$, let $V_i^r$ be a copy of $-V_i$ for $1\le r\le b_i$.
Now we define $W$ as follows:
\[
W:=V\bigcup_{\partial C_-} C\bigcup_{\partial C_+} \left(U\coprod \left(\coprod_i\coprod_{r=1}^{b_i}V_i^r\right)\right)
\]
where $\partial C_- := M(L)$ and $\partial C_+:=\coprod a_iM(K_i)$. See Figure~\ref{figure:Cobordism_W}. Note that $\partial W = M(K)$. 

\begin{figure}[H]
	\begin{tikzpicture}[x=1bp,y=1bp]
	\small
	\node [anchor=south west, inner sep=0mm] {\includegraphics[scale=0.5]{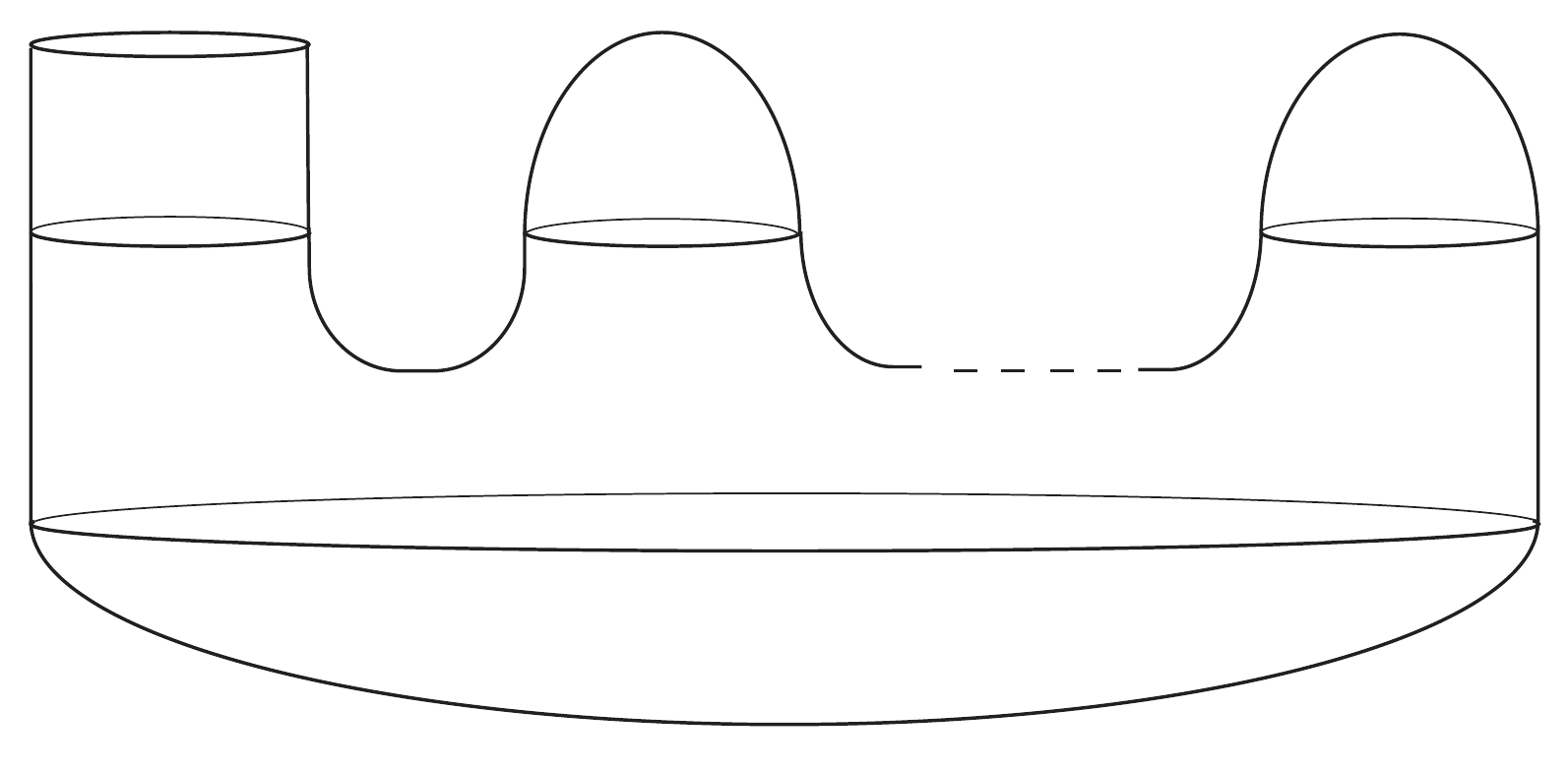}};
	\node at (120,20) {$V$};
	\node at (120,50) {$C$};
	\node at (25,90) {$U$};
	\node at (100,90) {$V_i^r$};
	\node at (208,90) {$V_i^r$};
	\node [left] at (5,35) {$M(L)$};
	\node [left] at (5,79) {$M(K_1)$};
	\node [left] at (5,105) {$M(K)$};
	\node [right] at (117,79) {$M(K_i)$};
	\node [right] at (225,79) {$M(K_i)$};
	\end{tikzpicture}
	\caption{Cobordism $W$}
	\label{figure:Cobordism_W}
\end{figure}
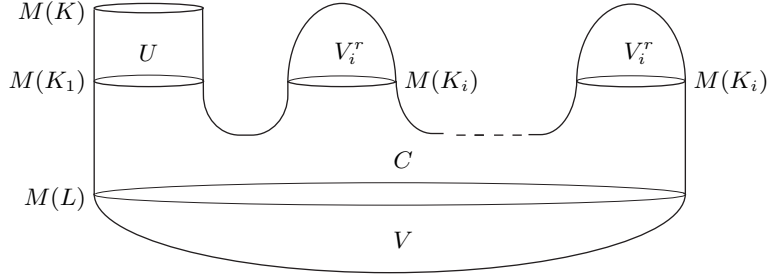

Let $\Gamma:=\pi_1W/(\pi_1W)^{(2)}_{p_1}$ as defined in Section~\ref{section:preliminaries} and $\phi\colon \pi_1W\to \Gamma$ the projection. By Lemma~\ref{lemma:amenable group}, the group $\Gamma$ is amenable and lies in Strebel's class $D(\Z_{p_1})$. By abuse of notation, let $\phi$ also denote the restriction of $\phi$ to subspaces of $W$. 

From the definition of $\rhot$-invariants given in Section~\ref{section:preliminaries}, we have $\rhot(M(K), \phi) = S_\Gamma(W)$. On the other hand, by Novikov additivity we have
\[
S_\Gamma(W) = S_\Gamma(V) + S_\Gamma(C) + S_\Gamma(U) + \sum_i\sum_{r=1}^{b_i} S_\Gamma(V_i^r).
\]
We compute each term of the right-hand side of the above equation.
\begin{enumerate}
	\item $S_\Gamma(V)=0$ by Theorem~\ref{theorem:obstruction} since $V$ is a slice disk exterior. 
	\item $S_\Gamma(C)=0$: since $H_2(C)\cong H_2(\partial_- C)$, it follows that $\Coker\{H_2(\partial_- C)\to H_2(C)\} = 0$. Now $\sign(C) = \lsign_\Gamma(C)=0$  as we have seen in the proof of Theorem~\ref{theorem:main-theorem}(3).
	\item Let $i>1$ and  $\epsilon_i:=-a_i/|a_i|$ Then, $S_\Gamma(V_i^r) = \epsilon_i\cdot \sum_{\ell=1}^m\rhot(M(J_i),\phi_\ell)$. Since $\eta_\ell\in \pi_1(M_K)^{(1)}$, we have $\phi(\eta_\ell)\in \Gamma^{(1)} = \pi_1W^{(1)}/(\pi_1W)^{(2)}_{p_1}$. Since $\Gamma^{(1)}$ injects into $(\pi_1W^{(1)}/\pi_1W^{(2)})\otimes \Z_{p_1}$, which is a $\Z_{p_1}$-vector space, we obtain that  $d_\ell=0\mbox{ or }p_1$. 
	
	If $d_\ell=0$, then $\phi_\ell$ is the trivial map and $\rhot(M(J_i),\phi_\ell) = 0$ by (2.5) in \cite[p.108]{Cochran-Orr-Teichner:2002-1}
	
	If $d_\ell=p_1$, then $\phi_\ell$ is a surjection to $\Z_{p_1}$, and by \cite[Lemma~8.7]{Cha-Orr:2009-1} and Lemma~\ref{lemma:J_i}(3), we obtain $\rhot(M(J_i),\phi_\ell) = \frac{1}{p_1} \sum_{r=0}^{p_1-1}\sigma_{J_i}(e^{2\pi r\sqrt{-1}/p_1}) =0$.
	
	Similarly, $S_\Gamma(V_1^r) = -\rhot(M(J_1),\phi_\ell) = 0\mbox{ or } -\frac{1}{p_1} \sum_{r=0}^{p_1-1}\sigma_{J_1}(e^{2\pi r\sqrt{-1}/p_1})$.
	
	Therefore, $\sum_i\sum_{r=1}^{b_i} S_\Gamma(V_i^r)\le  0$.
	\item $S_\Gamma(U)=-\sum_{\ell=1}^m\rhot(M(J_1),\phi_\ell)$, and similarly as in (3) above, $\rhot(M(J_1),\phi_\ell)=0$ if $d_\ell=0$ and $\frac{1}{p_1} \sum_{r=0}^{p_1-1}\sigma_{J_1}(e^{2\pi r\sqrt{-1}/p_1})$ if $d_\ell=p_1$. By Lemma~\ref{lemma:nontrivial} below, we can conclude that 
	$S_\Gamma(U)\le -\frac{1}{p_1} \sum_{r=0}^{p_1-1}\sigma_{J_1}(e^{2\pi r\sqrt{-1}/p_1})$.
\end{enumerate}
	\begin{lemma}\label{lemma:nontrivial}
		In (4) above, $d_\ell=p_1$ for some $\ell$.
	\end{lemma}
	\begin{proof}	
Let $I$ be the image of the map $i_*\colon H_1(M(K);\Z_{p_1}[t^{\pm 1}])\to H_1(W;\Z_{p_1}[t^{\pm 1}])$ where $i_*$ is induced from the inclusion map. Then, $\rank_{\Z_{p_1}} I \ge \frac12 \rank_{\Z_{p_1}} H_1(M(K);\Z_{p_1}[t^{\pm 1}])$. This can be seen by \cite[Theorem~5.2]{Kim:2016-1} observing that $W$ is a $(1)$-cylinder. This is the only place where we use the notion of $(n)$-cylinders, and since the arguments for showing $W$ is a $(1)$-cylinder is well-known for the experts, we give a brief proof that $W$ is a $(1)$-cylinder below. 
One may refer to \cite{Cochran-Kim:2004-1, Kim:2016-1} for the definition of an $(n)$-cylinder, but we will not use it below. 

By \cite[Lemma~4.1]{Kim:2016-1}, the 4-manifolds $V_i$ and $U$ are obtained as $(1)$-solutions and a $(1)$-cylinder, respectively. Since a $(1)$-solution is a $(1)$-cylinder (see \cite[Proposition~2.3]{Cochran-Kim:2004-1}), $V_i$ are also $(1)$-cylinders.  Since $\Coker\{H_2(\partial C)\to H_2(C)\}=0$ and $V$ is a slice disk exterior, the 4-manifolds $C$ and $V$ are also  $(1)$-cylinders. Since $W$ is a union of $(1)$-cylinders along common boundary components, one can easily show that $W$ is a $(1)$-cylinder following the arguments in the proof of \cite[Proposition~2.6]{Cochran-Kim:2004-1}.

Since $p_1>a_K$ by our choice of $p_1$, where $a_K$ is the top coefficient of $\Delta_K(t)$, we have $\rank_{\Z_{p_1}} H_1(M(K);\Z_{p_1}[t^{\pm 1}]) = \deg \Delta_K(t)$. Since $\Delta_K(t)$ is nontrivial, we have $\deg \Delta_K(t)\ge 2$. Therefore, $\rank_{\Z_{p_1}} I  \ge 1$, and hence $I\ne 0$. Since the $\eta_\ell$ generate $H_1(M(K);\Z_{p_1}[t^{\pm 1}])$, this implies that $i_*(\eta_\ell)\ne 0$ in $H_1(W;\Z_{p_1}[t^{\pm 1}])$ for some $\ell$. Since $\phi(\eta_\ell)\in \Gamma^{(1)}$ and $\Gamma^{(1)}$ injects into $(\pi_1W^{(1)}/\pi_1W^{(2)})\otimes \Z_{p_1}\cong H_1(W;\Z_{p_1}[t^{\pm 1}])$, which is a $\Z_{p_1}$-vector space, it follows that $\phi(\eta_\ell)$ has order 0 or $p_1$. But since $i_*(\eta_\ell)\ne 0$,  we have $\phi(\eta_\ell)\ne 0$. Therefore, $\phi(\eta_\ell)$ has order $p_1$. 
\end{proof}

Now by (1)-(4) and Lemma~\ref{lemma:J_i}(2), we conclude that 
\[
\rhot(M(K), \phi) = S_\Gamma(W) \le -\frac{1}{p_1} \sum_{r=0}^{p_1-1}\sigma_{J_1}(e^{2\pi r\sqrt{-1}/p_1}) < -C_K,
\]
which contradicts our choice of $C_K$. This proves Theorem~\ref{theorem:main-theorem}(2).

\begin{remark}
From the viewpoint of the solvable filtration $\{\F_n\}$ in \cite{Cochran-Orr-Teichner:1999-1}, for each $i$, $K_i\in \F_1$ since $\eta_\ell\in (\pi_1M(K))^{(1)}$ and $J_i\in \F_0$ (a knot with zero Arf invariant lies in $\F_0$). Also, the proof for Theorem~\ref{theorem:main-theorem}(2) is still available when $V$ is a (1.5)-solution, and  hence the knots $K_i$ are, in fact, linearly independent in $\F_1/\F_{1.5}$.
\end{remark}

\providecommand{\bysame}{\leavevmode\hbox to3em{\hrulefill}\thinspace}
\providecommand{\MR}{\relax\ifhmode\unskip\space\fi MR }
\providecommand{\MRhref}[2]{%
  \href{http://www.ams.org/mathscinet-getitem?mr=#1}{#2}
}
\providecommand{\href}[2]{#2}


\begin{thebibliography}{CHL11b}

\bibitem[Bao15]{Bao:2015-1}
Yuanyuan Bao, \emph{Polynomial splittings of {O}zsv\'ath and {S}zab\'o's
  {$d$}-invariant}, Topology Proc. \textbf{46} (2015), 309--322. \MR{3274180}

\bibitem[CFP14]{Cha-Friedl-Powell:2014-1}
Jae~Choon Cha, Stefan Friedl, and Mark Powell, \emph{Concordance of links with
  identical {A}lexander invariants}, Bull. Lond. Math. Soc. \textbf{46} (2014),
  no.~3, 629--642. \MR{3210718}

\bibitem[CG85]{Cheeger-Gromov:1985-1}
Jeff Cheeger and Mikhael Gromov, \emph{Bounds on the von {N}eumann dimension of
  {$L\sp 2$}-cohomology and the {G}auss-{B}onnet theorem for open manifolds},
  J. Differential Geom. \textbf{21} (1985), no.~1, 1--34. \MR{MR806699
  (87d:58136)}

\bibitem[Cha09]{Cha:2007-2}
Jae~Choon Cha, \emph{Structure of the string link concordance group and
  {H}irzebruch-type invariants}, Indiana Univ. Math. J. \textbf{58} (2009),
  no.~2, 891--927. \MR{MR2514393}

\bibitem[Cha14]{Cha:2010-1}
\bysame, \emph{Amenable {$L^2$}-theoretic methods and knot concordance}, Int.
  Math. Res. Not. IMRN (2014), no.~17, 4768--4803. \MR{3257550}

\bibitem[Cha16]{Cha:2014-1}
\bysame, \emph{A topological approach to {C}heeger-{G}romov universal bounds
  for von {N}eumann {$\rho$}-invariants}, Comm. Pure Appl. Math. \textbf{69}
  (2016), no.~6, 1154--1209. \MR{3493628}

\bibitem[CHL09]{Cochran-Harvey-Leidy:2009-1}
Tim~D. Cochran, Shelly Harvey, and Constance Leidy, \emph{Knot concordance and
  higher-order {B}lanchfield duality}, Geom. Topol. \textbf{13} (2009), no.~3,
  1419--1482. \MR{MR2496049 (2009m:57006)}

\bibitem[CHL11a]{Cochran-Harvey-Leidy:2009-3}
\bysame, \emph{2-torsion in the {$n$}-solvable filtration of the knot
  concordance group}, Proc. Lond. Math. Soc. (3) \textbf{102} (2011), no.~2,
  257--290. \MR{2769115 (2012c:57011)}

\bibitem[CHL11b]{Cochran-Harvey-Leidy:2009-2}
\bysame, \emph{Primary decomposition and the fractal nature of knot
  concordance}, Math. Ann. \textbf{351} (2011), no.~2, 443--508. \MR{2836668
  (2012k:57012)}

\bibitem[CK08]{Cochran-Kim:2004-1}
Tim~D. Cochran and Taehee Kim, \emph{Higher-order {A}lexander invariants and
  filtrations of the knot concordance group}, Trans. Amer. Math. Soc.
  \textbf{360} (2008), no.~3, 1407--1441 (electronic). \MR{MR2357701
  (2008m:57008)}

\bibitem[CO12]{Cha-Orr:2009-1}
Jae~Choon Cha and Kent~E. Orr, \emph{${L}^2$-signatures, homology localization,
  and amenable groups}, Comm. Pure Appl. Math. \textbf{65} (2012), 790--832.

\bibitem[COT03]{Cochran-Orr-Teichner:1999-1}
Tim~D. Cochran, Kent~E. Orr, and Peter Teichner, \emph{Knot concordance,
  {W}hitney towers and {$L\sp 2$}-signatures}, Ann. of Math. (2) \textbf{157}
  (2003), no.~2, 433--519. \MR{1 973 052}

\bibitem[COT04]{Cochran-Orr-Teichner:2002-1}
\bysame, \emph{Structure in the classical knot concordance group}, Comment.
  Math. Helv. \textbf{79} (2004), no.~1, 105--123. \MR{MR2031301 (2004k:57005)}

\bibitem[FQ90]{Freedman-Quinn:1990-1}
Michael~H. Freedman and Frank Quinn, \emph{Topology of 4-manifolds}, Princeton
  Mathematical Series, vol.~39, Princeton University Press, Princeton, NJ,
  1990. \MR{MR1201584 (94b:57021)}

\bibitem[Fre82]{Freedman:1982-1}
Michael~H. Freedman, \emph{The topology of four-dimensional manifolds}, J.
  Differential Geom. \textbf{17} (1982), no.~3, 357--453. \MR{MR679066
  (84b:57006)}

\bibitem[Jan15]{Jang:2015-1}
Hye~Jin Jang, \emph{2-torsion in the grope and solvable filtrations of knots},
  arXiv preprint arXiv:1502.04436 (2015).

\bibitem[Kaw96]{Kawauchi:1996-1}
Akio Kawauchi, \emph{A survey of knot theory}, Birkh\"auser Verlag, Basel,
  1996, Translated and revised from the 1990 Japanese original by the author.
  \MR{97k:57011}

\bibitem[Kim05a]{Kim:2005-2}
Se-Goo Kim, \emph{Polynomial splittings of {C}asson-{G}ordon invariants}, Math.
  Proc. Cambridge Philos. Soc. \textbf{138} (2005), no.~1, 59--78. \MR{2127228}

\bibitem[Kim05b]{Kim:2005-1}
Taehee Kim, \emph{An infinite family of non-concordant knots having the same
  {S}eifert form}, Comment. Math. Helv. \textbf{80} (2005), no.~1, 147--155.
  \MR{MR2130571 (2006a:57007)}

\bibitem[Kim16]{Kim:2016-1}
\bysame, \emph{Amenable signatures, algebraic solutions, and filtrations of the
  knot concordance group}, arXiv:1606.06807, 2016.

\bibitem[KK08]{Kim-Kim:2008-1}
Se-Goo Kim and Taehee Kim, \emph{Polynomial splittings of metabelian von
  {N}eumann rho-invariants of knots}, Proc. Amer. Math. Soc. \textbf{136}
  (2008), no.~11, 4079--4087. \MR{MR2425750 (2009e:57009)}

\bibitem[KK14]{Kim-Kim:2014-1}
\bysame, \emph{Splittings of von {N}eumann rho-invariants of knots}, J. Lond.
  Math. Soc. (2) \textbf{89} (2014), no.~3, 797--816. \MR{3217650}

\bibitem[KK16]{Kim-Kim:2016-1}
\bysame, \emph{Polynomial splittings of correction terms and doubly slice
  knots}, arXiv:1611.07656, to appear in J. Knot. Theor. Ramif., 2016.

\bibitem[KL16]{Kauffman-Lopes:2017-1}
Louis~H. Kauffman and Pedro Lopes, \emph{Infinitely many prime knots with the
  same alexander invariants}, arXiv:1604.02510, 2016.

\bibitem[Lev69a]{Levine:1969-2}
Jerome~P. Levine, \emph{Invariants of knot cobordism}, Invent. Math. 8 (1969),
  98--110; addendum, ibid. \textbf{8} (1969), 355. \MR{40 \#6563}

\bibitem[Lev69b]{Levine:1969-1}
\bysame, \emph{Knot cobordism groups in codimension two}, Comment. Math. Helv.
  \textbf{44} (1969), 229--244. \MR{39 \#7618}

\bibitem[Liv02]{Livingston:2002-1}
Charles Livingston, \emph{Seifert forms and concordance}, Geom. Topol.
  \textbf{6} (2002), 403--408 (electronic). \MR{MR1928840 (2003f:57019)}

\bibitem[Sto77]{Stoltzfus:1977-1}
Neal~W. Stoltzfus, \emph{Unraveling the integral knot concordance group}, Mem.
  Amer. Math. Soc. \textbf{12} (1977), no.~192, iv+91. \MR{0467764}

\end{thebibliography}
\end{document}